\documentclass[journal]{IEEEtran}
\usepackage{graphics} % for pdf, bitmapped graphics files
\usepackage{epsfig} % for postscript graphics files
\usepackage{amsmath} % assumes amsmath package installed
\usepackage{amsthm}
\usepackage{amssymb}  % assumes amsmath package installed
\usepackage{subfigure}
\usepackage{color} 
\usepackage{cite}

\newtheorem{theorem}{Theorem}%[section]
\newtheorem{remark}{Remark}%[section]
\newtheorem{example}{Example}%[section]
\newtheorem{definition}{Definition}

\ifCLASSINFOpdf

\else

\fi

\begin{document}
%
% paper title
% Titles are generally capitalized except for words such as a, an, and, as,
% at, but, by, for, in, nor, of, on, or, the, to and up, which are usually
% not capitalized unless they are the first or last word of the title.
% Linebreaks \\ can be used within to get better formatting as desired.
% Do not put math or special symbols in the title.
\title{Input-to-State Stability of Time-Delay Systems with Delay-Dependent Impulses}
%
%
% author names and IEEE memberships
% note positions of commas and nonbreaking spaces ( ~ ) LaTeX will not break
% a structure at a ~ so this keeps an author's name from being broken across
% two lines.
% use \thanks{} to gain access to the first footnote area
% a separate \thanks must be used for each paragraph as LaTeX2e's \thanks
% was not built to handle multiple paragraphs
%

\author{Xinzhi~Liu~%~\IEEEmembership{Member,~IEEE,}
        and~Kexue~Zhang%~\IEEEmembership{Member,~IEEE}
        %and~Jane~Doe,~\IEEEmembership{Life~Fellow,~IEEE}% <-this % stops a space
\thanks{X. Liu and K. Zhang are with the Department
of Applied Mathematics, University of Waterloo, Waterloo, Ontario N2L 3G1, Canada (e-mail: xinzhi.liu@uwaterloo.ca; kexue.zhang@uwaterloo.ca).}% <-this % stops a space
%\thanks{J. Doe and J. Doe are with Anonymous University.}% <-this % stops a space
%\thanks{Manuscript received April 19, 2005; revised August 26, 2015.}
}

\maketitle

% As a general rule, do not put math, special symbols or citations
% in the abstract or keywords.
\begin{abstract}
This paper studies input-to-state stability (ISS) of general nonlinear time-delay systems subject to delay-dependent impulse effects. Sufficient conditions for ISS are constructed by using the method of Lyapunov functionals. It is shown that, when the continuous dynamics are ISS but the discrete dynamics governing the delay-dependent impulses are not, the impulsive system as a whole is ISS provided the destabilizing impulses do not occur too frequently. On the contrary, when the discrete dynamics are ISS but the continuous dynamics are not, the delayed impulses must occur frequently enough to overcome the destabilizing effects of the continuous dynamics so that the ISS can be achieved for the impulsive system. Particularly, when the discrete dynamics are ISS and the continuous dynamics are also ISS or just stable for the zero input, the impulsive system is ISS for arbitrary impulse time sequences. { Compared with the existing results on impulsive time-delay systems, the obtained ISS criteria are more general in the sense that these results are applicable to systems with delay dependent impulses while the existing ones are not. Moreover, when consider time-delay systems with delay-free impulses, our result for systems with unstable continuous dynamics and stabilizing impulses is less conservative than the existing ones, as a weaker condition on the upper bound of impulsive intervals is obtained.} To demonstrate the theoretical results, we provide two examples with numerical simulations, in which distributed delays and discrete delays in the impulses are considered respectively.
\end{abstract}

% Note that keywords are not normally used for peerreview papers.

% For peer review papers, you can put extra information on the cover
% page as needed:
% \ifCLASSOPTIONpeerreview
% \begin{center} \bfseries EDICS Category: 3-BBND \end{center}
% \fi
%
% For peerreview papers, this IEEEtran command inserts a page break and
% creates the second title. It will be ignored for other modes.
\IEEEpeerreviewmaketitle

\section{Introduction}
% The very first letter is a 2 line initial drop letter followed
% by the rest of the first word in caps.
% 
% form to use if the first word consists of a single letter:
% \IEEEPARstart{A}{demo} file is ....
% 
% form to use if you need the single drop letter followed by
% normal text (unknown if ever used by the IEEE):
% \IEEEPARstart{A}{}demo file is ....
% 
% Some journals put the first two words in caps:
% \IEEEPARstart{T}{his demo} file is ....
% 
% Here we have the typical use of a "T" for an initial drop letter
% and "HIS" in caps to complete the first word.
%\IEEEPARstart{I}{mpulsive} systems, delay-dependent impulses, ISS and research on impulsive systems, what we do in this paper and contribution, organization. 
% You must have at least 2 lines in the paragraph with the drop letter
% (should never be an issue)
%I wish you the best of success.
%
%\hfill mds
%% 
%\hfill August 26, 2015
%
%\subsection{Subsection Heading Here}
%Subsection text here.
%
%% needed in second column of first page if using \IEEEpubid
%%\IEEEpubidadjcol
%
%\subsubsection{Subsubsection Heading Here}
%Subsubsection text here.

\IEEEPARstart{I}{mpulsive} system is a dynamical system (modeled by impulsive differential equations) that combines continuous-time evolution and abrupt state changes or resets at a sequence of times. Due to their wide applications in consensus of multi-agent systems \cite{ICM-AM-AG-AMG:2016}, synchronization of dynamical networks and applications to secure communication \cite{TY-LOC:1997}, infection disease control \cite{SG-LC-JJN-AT:2006}, spacecraft maneuvers \cite{TC:1991}, etc., extensive studies have been done on stability analysis and control problems of impulsive systems (see, e.g., \cite{BMM-EYR:2012,NP-AMS:1995,WMH-VSC-SGN:2006,STZ-ANS:1997}). { Since time-delay is unavoidable in sampling, processing, and transmission of the impulse information in the system, recent years have witnessed a rapid progress in stability analysis and control of dynamical systems with delay-dependent impulses (see, e.g., \cite{XL-KZ:2018,XL-KZ-WCX:2016,AK-XL-XS:2009,XL-XZ-SS:2017}). }

The notion of input-to-state stability (ISS), introduced in \cite{EDS:1989}, has been proved very useful in stability analysis and control of dynamical systems. This notion characterize the effects of external inputs to the stability of control systems. The ISS properties have been investigated for various types of control systems (see, e.g., \cite{ZPJ-YW:2001,LV-DC-DL:2007,DN-ART:2004,ART-LM-DN:2003,MA-ART:2012,ML-WPMHH-ART:2009,SD-AM:2013,NY-PP-MD:2008,ENS-JPP:1999}). In terms of dynamical systems with impulse effects, the ISS notion is extended to impulsive systems (without time-delay) in \cite{JPH-DL-ART:2008}, and then to impulsive systems with time-delay in the continuous time dynamics in \cite{WHC-WXZ:2009}. More recently, the work of \cite{JL-XL-WCX:2011} generalizes the ISS notion to impulsive time-delay systems with switchings, and sufficient conditions for ISS are obtained using the method of Lyapunov functionals. In \cite{XL-XZ-SS:2017}, the effect of delayed impulses on the ISS property of nonlinear delay-free systems is studied by the Lyapunov method. Up to now, numerous researchers have investigated the ISS properties of impulsive systems with time-delay (see, e.g., \cite{XL-XZ-SS:2017,JPH-DL-ART:2008,WHC-WXZ:2009,JL-XL-WCX:2011,XMS-WW:2012,SD-MK-AM-LN:2012,XW-YT-WZ:2016,YW-RW-JZ:2013,XZ-XL:2016}). However, to our best knowledge, time-delay in only considered in either the continuous dynamics of impulsive systems or just the impulses in the existing literature, and very few ISS results has been reported for systems with time-delay effects in both the continuous dynamics and the impulses. { Due to the ubiquity of time-delay, examples of such systems can be easily found in impulsive stabilization of neural networks \cite{XL-KZ-WCX:2016}, impulsive consensus of networked multi-agent systems \cite{XL-KZ-WCX:2019}, impulsive master-slave synchronization \cite{JS-QLH-XJ:2008}, etc.}

Motivated by the above discussion, we focus on nonlinear time-delay systems with delay-dependent impulses. The objective of this paper is to construct ISS criteria for such systems using the method of Lyapunov functionals. { Compared with the existing results, especially the representative results in \cite{XL-XZ-SS:2017,JPH-DL-ART:2008,WHC-WXZ:2009,JL-XL-WCX:2011,XMS-WW:2012}, the main contribution is that our ISS results are constructed for nonlinear systems with time-delay effects in both the continuous dynamics and the impulses, while the results in \cite{JPH-DL-ART:2008,WHC-WXZ:2009,JL-XL-WCX:2011,XMS-WW:2012} cannot be applied to systems with delay-dependent impulses and the ISS results in \cite{XL-XZ-SS:2017,JPH-DL-ART:2008} is not applicable to impulsive systems with time-delay in the continuous evolution.} Furthermore, our ISS result for systems with unstable continuous dynamics and ISS impulses is less conservative than the ones in \cite{JL-XL-WCX:2011,XMS-WW:2012} in the sense that our result can be applied to systems with a larger class of impulse time sequences.

The rest of the paper is organized as follows. Section \ref{Sec2} gives the basic notation and definitions, and formulate a general nonlinear impulsive systems with time-delay and external input. Sufficient conditions for ISS of the impulsive time-delay systems are constructed in Section \ref{Sec3} by using the method of Lyapunov functionals. Two numerical examples are provided in Section \ref{Sec4} to demonstrate the main results. The paper is concluded by Section \ref{Sec4}, where some future research directions are pointed out.

\section{Preliminaries}\label{Sec2}

Let $\mathbb{N}$ denote the set of positive integers, $\mathbb{R}$ the set of real numbers, $\mathbb{R}^+$ the set of nonnegative reals, and $\mathbb{R}^n$ the $n$-dimensional real space equipped with the Euclidean norm denoted by $\|\cdot\|$. For $a,b\in \mathbb{R}$ with $b>a$, denote $\mathcal{PC}([a,b],\mathbb{R}^n)$ the set of piecewise right continuous functions $\varphi:[a,b]\rightarrow\mathbb{R}^n$, and $\mathcal{PC}([a,\infty),\mathbb{R}^n)$ the set of functions $\phi:[a,\infty)\rightarrow\mathbb{R}^n$ satisfying $\phi|_{[a,b]}\in \mathcal{PC}([a,b],\mathbb{R}^n)$ for all $b>a$, where $\phi|_{[a,b]}$ is a restriction of $\phi$ on interval $[a,b]$. Given $r>0$, the linear space $\mathcal{PC}([-r,0],\mathbb{R}^n)$ is equipped with a norm defined by $\|\varphi\|_r:=\sup_{s\in[-r,0]}\|\varphi(s)\|$ for $\varphi\in \mathcal{PC}([-r,0],\mathbb{R}^n)$. For simplicity, we use $\mathcal{PC}$ to represent $\mathcal{PC}([-r,0],\mathbb{R}^n)$. Given $x\in \mathcal{PC}([-r,\infty),\mathbb{R}^n)$ and for each $t\in\mathbb{R}^+$, we define $x_t\in\mathcal{PC}$ as $x_t(s):=x(t+s)$ for $s\in [-r,0]$.

Consider the following nonlinear time-delay impulsive system:
\begin{eqnarray}\label{sys}
\left\{\begin{array}{ll}
\dot{x}(t)=f(t,x_t,w(t)), & t\not=t_k,~k\in\mathbb{N}\cr
\Delta x(t)=I_k(t,x_{t^-},w(t^-)), & t=t_k,~k\in\mathbb{N}\cr
x_{t_0}=\varphi,
\end{array}\right.
\end{eqnarray}
where $x(t)\in\mathbb{R}^n$ is the system state; $w\in \mathcal{PC}([t_0,\infty),\mathcal{R}^m)$ is the input function; $\varphi\in\mathcal{PC}$ is the initial function; $f,I_k:\mathbb{R}^+\times\mathcal{PC}\times\mathbb{R}^m\rightarrow\mathbb{R}^n$ satisfy $f(t,0,0)=I_k(t,0,0)=0$ for all $k\in\mathbb{N}$; $\{t_1,t_2,t_3,...\}$ is a strictly increasing sequence such that $t_k\rightarrow\infty$ as $t\rightarrow\infty$; $\Delta x(t):=x(t^+)-x(t^-)$ where $x(t^+)=\lim_{s\rightarrow t^+} x(s)$ and $x(t^-)=\lim_{s\rightarrow t^-} x(s)$ (similarly, $w(t^-)=\lim_{s\rightarrow t^-} w(s)$); $x_{t^-}$ is defined as $x_{t^-}(s)=x(t+s)$ if $s\in[-r,0)$ and $x_{t^-}(s)=x(t^-)$ if $s=0$. Given $w\in \mathcal{PC}([t_0,\infty),\mathcal{R}^m)$, define $g(t,\phi)=f(t,\phi,w(t))$ and assume $g$ satisfies all the necessary conditions in \cite{GB-XL:1999} so that, for any initial condition $\varphi\in\mathcal{PC}$, system \eqref{sys} has a unique solution $x(t,t_0,\varphi)$ that exists in a maximal interval $[t_0-r,t_0+\beta)$, where $0<\beta\leq \infty$.

Before giving the formal ISS definition for system \eqref{sys}, we introduce the following function classes. A continuous function $\alpha:\mathbb{R}^+\rightarrow\mathbb{R}$ is said to be of class $\mathcal{K}$ and we write $\alpha\in\mathcal{K}$, if $\alpha$ is strictly increasing and $\alpha(0)=0$. If $\alpha$ also unbounded, we say that $\alpha$ is of class $\mathcal{K}_{\infty}$ and we write $\alpha\in\mathcal{K}_{\infty}$. A function $\beta:\mathbb{R}^+\times\mathbb{R}^+ \rightarrow\mathbb{R}^+$ is said to be of class $\mathcal{KL}$ and we write $\beta\in \mathcal{KL}$, if $\beta(\cdot,t)\in\mathcal{K}$ for each $t\in\mathbb{R}^+$ and $\beta(s,t)$ decreases to $0$ as $t\rightarrow \infty$ for each $s\in \mathbb{R}^+$. Now we are in the position to state the ISS definition for system \eqref{sys}.

\begin{definition}
System \eqref{sys} is said to be uniformly input-to-state { stable} (ISS) over a certain class $\ell$ of admissible impulse time sequences, if there exist functions $\beta\in\mathcal{KL}$ and $\gamma\in\mathcal{K}_{\infty}$, independent of the choice of the sequences in $\ell$, such that, for each initial condition $\varphi\in\mathcal{PC}$ and input function $w\in\mathcal{PC}([t_0,\infty),\mathbb{R}^m)$, the corresponding solution to \eqref{sys} exists globally and satisfies
$$\|x(t)\|\leq \beta(\|\varphi\|_r,t-t_0)+\gamma\bigg(\sup_{s\in[t_0,t]}\|w(s)\|\bigg), \mathrm{~for~all~} t\geq t_0.$$
\end{definition}

To study the ISS properties of system \eqref{sys}, we next introduce two function classes related to the Lyapunov functional candidates. A function $v:\mathbb{R}^+\times\mathbb{R}^n\rightarrow \mathbb{R}^+$ is said to be of class $\nu_0$ and we write $v\in\nu_0$, if, for each $x\in\mathcal{PC}(\mathbb{R}^+,\mathbb{R}^n)$, the composite function $t\mapsto v(t,x(t))$ is also in $\mathcal{PC}(\mathbb{R}^+,\mathbb{R}^n)$ and can be discontinuous at some $t'\in\mathbb{R}^+$ only when $t'$ is a discontinuity point of $x$. A functional $v:\mathbb{R}^+\times \mathcal{PC}\rightarrow\mathbb{R}^n$ is said to be of class $\nu^*_0$ and we write $v\in \nu^*_0$, if, for each function $x\in\mathcal{PC}([-r,\infty),\mathbb{R}^n)$, { the composite function $t\mapsto v(t,x_t)$ is continuous in $t$ for all $t\geq 0$}. To analyze the continuous dynamics of system \eqref{sys}, we introduce the upper right-hand derivative of the Lyapunov functional candidate $V(t,x_t)$ with respect to system \eqref{sys}:
$$\mathrm{D}^+V(t,\phi)=\limsup_{h\rightarrow 0^+}\frac{1}{h}[V(t+h,x_{t+h}(t,\phi))-V(t,\phi)],$$
where $x(t,\phi)$ is a solution to \eqref{sys} satisfying $x_t=\phi$.
, that is, $x(s):=x(t,\phi)(s)$ is a solution to the initial value problem: $\left\{\begin{array}{ll}
\dot{x}(s)=f(s,x_s,w(s))\cr
x_{s_0}=\varphi
\end{array}\right.$ for $s\in[t,t+h)$ and $s_0=t$, where $h$ is a small enough positive number such that no impulse time lies in the interval $(t,t+h)$.

\section{Sufficient Conditions for ISS}\label{Sec3}

In this section, we establish several ISS results for system \eqref{sys} over the following three types of impulse time sequences, respectively: (i) $\ell_{\textrm{inf}}(\delta)$, the class of impulse time sequences satisfying $\inf_{k\in\mathbb{N}}\{t_k-t_{k-1}\}\geq \delta$; (ii) $\ell_{\textrm{sup}}(\delta)$, the class of impulse time sequences satisfying $\sup_{k\in\mathbb{N}}\{t_k-t_{k-1}\}\leq \delta$; (iii) $\ell_{\textrm{all}}$, the set containing all the possible impulse time sequences. We first introduce two results for ISS of system \eqref{sys} with stable continuous dynamics and unstable discrete dynamics.

\begin{theorem}
\label{th1}
Assume that there exist $V_1\in \nu_0$, $V_2\in \nu^*_0$, functions $\alpha_1,\alpha_2,\alpha_3,\chi\in\mathcal{K}_{\infty}$ and constants $\mu>0$, $\rho_1\geq 1$ and $\rho_2\geq 0$, such that, for all $t\in \mathbb{R}^+$, $x\in\mathbb{R}^n$, $y\in\mathbb{R}^m$ and $\phi\in \mathcal{PC}$, 
\begin{itemize}
\item[(i)] $\alpha_1(\|x\|)\leq V_1(t,x)\leq \alpha_2(\|x\|)$ and $0\leq V_2(t,\phi)\leq \alpha_3(\|\phi\|_{ r})$;

\item[(ii)] $\mathrm{D}^+ V(t,\phi)\leq -\mu V(t,\phi) + \chi(\|w(t)\|)$, where $V(t,\phi)=V_1(t,\phi(0))+V_2(t,\phi)$;

\item[(iii)] $V_1(t,\phi(0)+I_k(t,\phi,y))\leq \rho_1 V_1(t^-,\phi(0))+\rho_2 \sup_{s\in[-r,0]}\{V_1(t^-+s,\phi(s))\}+\chi(\|y\|)$;

\item[(iv)] $\ln \rho <\mu\delta$ where $\rho:=\rho_1+\rho_2 e^{\mu r}$.
\end{itemize}
Then system (\ref{sys}) is uniformly ISS over $\ell_{\textrm{inf}}(\delta)$.
\end{theorem}

\begin{proof}
Condition (iv) implies that there exists a small enough constant $\lambda>0$ so that $\ln\rho\leq(\mu-\lambda)\delta$. Then we can choose a large enough constant $c>0$ such that $c\rho e^{-\mu\delta} +\rho/\mu\leq c$ and $-\mu+1/c+\lambda\leq 0$. Let $x$ be a solution of \eqref{sys}, and set $v_1(t):=V_1(t,x(t))$, $v_2(t):=V_2(t,x_t)$, and $v(t):=v_1(t)+v_2(t)$. By mathematical induction, we shall prove that
\begin{align}\label{eqn-u}
v(t)e^{\lambda(t-t_0)} \leq & \alpha(\|\varphi\|_r) +c e^{\lambda(t-t_0)}\bar{\chi}(t)\cr
                            & +\sum_{t_0<t_k\leq t} e^{\lambda(t_k-t_0)} \chi(\|w(t^-_k)\|),
\end{align}
where $\alpha(\|\varphi\|_r):=\alpha_2(\|\varphi(0)\|)+\alpha_3(\|\varphi\|_{ r})$ and $\bar{\chi}(t)=\chi(\sup_{s\in[t_0,t]}\|w(s)\|)$. For convenience, denote the RHS of \eqref{eqn-u} as $u(t)$.

For $k\geq 1$, condition (ii) on $[t_k,t_{k+1})$ implies $e^{\mu t}v(t)-e^{\mu t_k} v(t_k) \leq \int^t_{t_k} e^{\mu s} \chi(\|w(s)\|) \mathrm{d}s$, which gives
\begin{equation}\label{eqn-v}
v(t)\leq e^{-\mu(t-t_k)} v(t_k) +\frac{1}{\mu} \bar{\chi}(t),
\end{equation}
for $t\in [t_k,t_{k+1})$ and $k\in \mathbb{Z}^+$. For $t\in [t_0,t_1)$, multiplying both sides of \eqref{eqn-v} with $e^{\lambda(t-t_0)}$ and using condition (i) and the fact that $1/\mu<c$, we conclude that $v(t)e^{\lambda(t-t_0)}\leq \alpha(\|\varphi\|) + c e^{\lambda(t-t_0)} \bar{\chi}(t)$ which shows that \eqref{eqn-u} holds for $t\in [t_0,t_1)$.

Now suppose that \eqref{eqn-u} holds for $t\in [t_0,t_m)$ where $m\geq 1$. We shall prove \eqref{eqn-u} is true on $[t_m,t_{m+1})$. To do this, we firstly conduct the following estimation:
\begin{align}\label{eqn-tk-}
      &\rho v(t^-_m) e^{\lambda(t_m-t_0)}\cr
\leq  &\rho e^{-(\mu-\lambda)(t_m-t_{m-1})} \alpha(\|\varphi\|)\cr
      &+ [c \rho e^{-\mu(t_m-t_{m-1})} + \frac{\rho}{\mu}] e^{\lambda(t_m-t_0)} \bar{\chi}(t^-_{m})\cr
      &+ \rho e^{-(\mu-\lambda)(t_m-t_{m-1})}  \sum_{t_0<t_k\leq t_{m-1}} e^{\lambda(t_k-t_0)}\chi(\|w(t^-_k)\|)\cr
\leq  &u(t^-_m).
\end{align}
For the first inequality of \eqref{eqn-tk-}, we used \eqref{eqn-v} and then \eqref{eqn-u}. For the second inequality of \eqref{eqn-tk-}, we used the facts that $t_m-t_{m-1}\geq \delta$, $\rho e^{-(\mu-\lambda)\delta}\leq 1$ and $c\rho e^{-\mu\delta} +\rho/\mu\leq c$. Next, we will show that \eqref{eqn-u} is true for $t=t_m$. We start with making the claim that for $s\in[-r,0]$, we have
\begin{equation}\label{eqn-claim}
\rho v(t^-_m+s)e^{\lambda(t_m+s-t_0)}\leq e^{(\mu-\lambda)r} u(t^-_m).
\end{equation}
Without loss of generality, suppose $t_m+s\geq t_0$ for all $s\in[-r,0]$, then, for a given $s\in[-r,0]$, there exists an integer $j$ ($0\leq j\leq m-1$) such that $t_m+s\in [t_j,t_{j+1})$ and
\begin{align}
& \rho v(t^-_m+s) e^{\lambda(t_m+s-t_0)}\cr
\leq & \rho e^{-(\mu-\lambda)(t_m+s-t_j)} u(t_j) + \frac{\rho}{\mu} e^{\lambda(t_m+s-t_0)} \bar{\chi}(t^-_m+s)\cr
\leq & e^{(\mu-\lambda)r}\rho e^{-(\mu-\lambda)(t_{j+1}-t_j)} u(t_j) + \frac{\rho}{\mu} e^{\lambda(t_m+s-t_0)} \bar{\chi}(t^-_m+s)\cr
\leq & e^{(\mu-\lambda)r} u(t^-_m),
\end{align}
which implies \eqref{eqn-claim} is true for all $s\in[-r,0]$. Combining this with condition (iii) and the fact that $\rho_1\geq 1$, we conclude that
\begin{align}\label{eqn-tm}
& v(t_m) e^{\lambda(t_m-t_0)}\cr
\leq & [\rho_1 v_1(t^-_m) + \rho_2 \sup_{-r\leq s\leq 0}\{v_1(t^-_m+s)\} + \chi(\|w(t^-_m)\|) \cr
     & + v_2(t^-_m)] e^{\lambda(t_m-t_0)}\cr
\leq & [\rho_1 v(t^-_m) + \rho_2 \sup_{-r\leq s\leq 0}\{v(t^-_m+s)\} + \chi(\|w(t^-_m)\|)] \cr
     & \times e^{\lambda(t_m-t_0)}\cr
\leq & \rho_1 v(t^-_m) e^{\lambda(t_m-t_0)} + \rho_2 e^{\lambda r} \sup_{-r\leq s\leq 0}\{v(t^-_m+s) e^{\lambda(t_m+s-t_0)}\} \cr
     & + \chi(\|w(t^-_m)\|) e^{\lambda(t_m-t_0)}\cr
\leq & \frac{\rho_1+\rho_2 e^{\mu r}}{\rho} u(t^-_m) + \chi(\|w(t^-_m)\|) e^{\lambda(t_m-t_0)}\cr
  =  & u(t_m),
\end{align}
which implies \eqref{eqn-u} holds for $t=t_m$. We now show that \eqref{eqn-u} is true on $(t_m,t_{m+1})$ by contradiction. Suppose that there exists a $t\in(t_m,t_{ m+1})$ such that $v(t)e^{\lambda(t-t_0)}>u(t)$, and then define $t^*=\inf\{t\in(t_m,t_{ m+1})\mid v(t)e^{\lambda(t-t_0)}>u(t)\}$. Combining \eqref{eqn-tm} and the continuities of $v(t)$ and $u(t)$ on $(t_m,t_{m+1})$, we conclude that $v(t^*)e^{\lambda(t^*-t_0)}=u(t^*)$, which, by the definition of $u(t)$, implies $v(t^*) >c \bar{\chi}(t^*)$. In view of this, condition (ii) and the fact that $-\mu +\frac{1}{c}+\lambda \leq 0$, it follows that
\begin{align}\label{eqn-dini}
& \mathrm{D}^+[v(t^*)e^{\lambda(t^*-t_0)}]\cr
\leq & [-\mu v(t^*) +\chi(\|w(t^*)\|)+\lambda v(t^*)]e^{\lambda(t^*-t_0)}\cr
  <  & (-\mu +\frac{1}{c}+\lambda) v(t^*) e^{\lambda(t^*-t_0)}\leq 0,
\end{align}
which means $v(t)e^{\lambda(t-t_0)}$ is strictly decreasing at $t=t^*$. This contradicts how $t^*$ is defined. Hence, \eqref{eqn-u} holds on $(t_m,t_{m+1})$. By induction, we conclude that \eqref{eqn-u} is true for all $t\geq t_0$. The ISS estimation can be conducted from \eqref{eqn-u} by standard arguments. The details are the same as that in \cite{JL-XL-WCX:2011} and thus omitted. Boundedness of the solution to \eqref{sys} follows from this estimate, which then implies the solution's global existence (see \cite{GB-XL:1999}).
\end{proof}

{ 
\begin{remark}\label{remark-th1}
Compared with Theorem 3.1 in \cite{JL-XL-WCX:2011} for system \eqref{sys}, the main difference is that the time-delay is considered in the impulses in our result. To be more specific, at the impulse time $t=t_k$, state $x(t^-_k)$ depends on $x_{t^-_k}$ which characterizes the states at some history moments. Therefore, in condition (iii) of Theorem \ref{th1}, the impulse effects on $V_1$ are described by a delay-independent part and a delay-dependent part which are associated with parameters $\rho_1$ and $\rho_2$, respectively. It is worth noting that $\rho_1\geq 1$ implies the delay-independent impulse part destabilizes the system. On the other hand, if the state jumps are independent of the time-delay (e.g., $\Delta x(t_k)=I_k(t_k,x(t_k^-),w(t_k^-))$), then $V_1$ depends only on state $x(t^-_k)$ (that is, $V_1$ is independent of the delayed states, which implies $\rho_2=0$) and Theorem \ref{th1} reduces to  Theorem 3.1 in \cite{JL-XL-WCX:2011} for system \eqref{sys}.

%In condition (iii) of Theorem \ref{th1}, the impulse effects on $V_1$ are described by a delay-independent part and a delay-dependent part which are associated with parameters $\rho_1$ and $\rho_2$, respectively. It is worth noting that $\rho_1\geq 1$ implies the delay-independent impulse part destabilizes the system. Compared with Theorem 3.1 in \cite{JL-XL-WCX:2011} for system \eqref{sys}, the main difference is that the time-delay is considered in the impulses in our result. If $\rho_2=0$, then Theorem \ref{th1} reduces to  Theorem 3.1 in \cite{JL-XL-WCX:2011} for system \eqref{sys}.
\end{remark}
}

The second result is concerned with system \eqref{sys} in the case when the impulses are potentially destabilizing but the delay-independent impulse parts are stabilizing (i.e., $\rho_1<1$).

\begin{theorem}
\label{th2}
Assume that there exist $V_1\in \nu_0$, $V_2\in \nu^*_0$, functions $\alpha_1,\alpha_2,\alpha_3,\chi\in\mathcal{K}_{\infty}$ and constants $\mu>0$, $1>\rho_1\geq 0$ and $\rho_2\geq 0$, such that, for all $t\in \mathbb{R}^+$, $x\in\mathbb{R}^n$, $y\in\mathbb{R}^m$ and $\phi\in \mathcal{PC}$, conditions (i), (ii) and (iii) of Theorem \ref{th1} are satisfied, and there exists a positive constant $\kappa$ such that $V_2(t,\phi)\leq \kappa \sup_{s\in[-r,0]}\{V_1(t+s,\phi(s))\}$. If we further assume that condition (iv) of Theorem \ref{th1} holds with $\rho:=\rho_1+[\rho_2+(1-\rho_1)\kappa] e^{\mu r}$ and $\rho_1+\rho_2+(1-\rho_1)\kappa\geq 1$, then system (\ref{sys}) is uniformly ISS over $\ell_{\textrm{inf}}(\delta)$.
\end{theorem}

\begin{proof}
The proof is essentially identical to that of Theorem \ref{th1}. The main difference is to replaced the following estimate of $v(t_m)$ in \eqref{eqn-tm}:
\begin{align}
 & v(t_m) \cr
 \leq & \rho_1 v_1(t^-_m) + \rho_2 \sup_{-r\leq s\leq 0}\{v_1(t^-_m+s)\} + \chi(\|w(t^-_m)\|) \cr
     & + (\rho_1 +1-\rho_1)v_2(t^-_m)\cr
\leq & \rho_1 v(t^-_m) + [\rho_2+(1-\rho_1)\kappa] \sup_{-r\leq s\leq 0}\{v(t^-_m+s)\} \cr
     & + \chi(\|w(t^-_m)\|).\nonumber
\end{align}
\end{proof}

We next introduce an ISS result for system \eqref{sys} for the case when the impulses are stabilizing but the continuous dynamics can be unstable.

\begin{theorem}
\label{th3}
Assume that there exist $V_1\in \nu_0$, $V_2\in \nu^*_0$, functions $\alpha_1,\alpha_2,\alpha_3,\chi\in\mathcal{K}_{\infty}$ and constants $\mu>0$, $\kappa>0$, $1>\rho_1\geq 0$ and $\rho_2\geq 0$, such that, for all $t\in \mathbb{R}^+$, $x\in\mathbb{R}^n$, and $\phi\in \mathcal{PC}$, 
\begin{itemize}
\item[(i)] $\alpha_1(\|x\|)\leq V_1(t,x)\leq \alpha_2(\|x\|)$ and $0\leq V_2(t,\phi)\leq \alpha_3(\|\phi\|_{ r})$;

\item[(ii)] $\mathrm{D}^+ V(t,\phi)\leq \mu V(t,\phi) + \chi(\|w(t)\|)$, where $V(t,\phi)=V_1(t,\phi(0))+V_2(t,\phi)$;

\item[(iii)] $V_1(t,\phi(0)+I_k(t,\phi,w(t^-)))\leq \rho_1 V_1(t^-,\phi(0))+\rho_2 \sup_{s\in[-r,0]}\{V_1(t^-+s,\phi(s))\}+ \chi(\sup_{s\in[-r,0]}\|w(t^-+s)\|)$;

\item[(iv)] $V_2(t,\phi)\leq \kappa \sup_{s\in[-r,0]}\{V_1(t+s,\phi(s))\}$;

\item[(v)] $\ln[\rho_1+\rho_2+(1-\rho_1)\kappa]  < -\mu \delta$.
\end{itemize}
Then system (\ref{sys}) is uniformly ISS over $\ell_{\textrm{sup}}(\delta)$.
\end{theorem}

\begin{proof}
We conclude from condition (v) that there exists a small enough constant $\lambda>0$ so that $\ln(\rho_1+[\rho_2+(1-\rho_1)\kappa] e^{\lambda r}) \leq -(\mu+\lambda)\delta$. Denote $\rho:=\rho_1+[\rho_2+(1-\rho_1)\kappa] e^{\lambda r}$, then we have $\rho e^{(\mu+\lambda)\delta}\leq 1$, which implies $\rho e^{\mu \delta}< 1$. This further implies that there exists a large enough constant $c>0$ such that $c \rho e^{\mu\delta} +{e^{\mu\delta}}/{\mu} \leq c$. Let $M:=e^{(\mu+\lambda)\delta}$, and we shall show that 
\begin{align}\label{eqn-u2}
v(t)e^{\lambda(t-t_0)} \leq & M \alpha(\|\varphi\|_r) +c e^{\lambda(t-t_0)}\bar{\chi}(t)\cr
                            & + M \sum_{t_0<t_k\leq t} e^{\lambda(t_k-t_0)} \hat{\chi}(t^-_k),
\end{align}
where $\alpha$ and $\bar{\chi}$ are the same as in the proof of Theorem \ref{th1}, and $\hat{\chi}(t^-_k):=\chi(\sup_{s\in[-r,0]}\|w(t^-_k+s)\|)$. For convenience, let $u(t)$ represent the RHS of \eqref{eqn-u2}. Similar to the estimate of \eqref{eqn-v}, we conclude from condition (ii) that 
\begin{equation}\label{eqn-v2}
v(t) \leq e^{\mu(t-t_k)} v(t_k) +\frac{e^{\mu\delta}}{\mu} \bar{\chi}(t),
\end{equation}
for $t\in [t_k,t_{k+1})$ and $k\geq 1$. Then, on $[t_0,t_1)$, we have
\begin{align}\label{eqn-k=1}
& v(t)e^{\lambda(t-t_0)}\cr
\leq & e^{(\mu+\lambda)(t-t_0)} v(t_0) + \frac{e^{\mu\delta}}{\mu} e^{\lambda(t-t_0)} \bar{\chi}(t)\cr
\leq & M \alpha(\|\varphi\|) + c e^{\lambda(t-t_0)} \bar{\chi}(t),
\end{align}
which means \eqref{eqn-u2} holds on $[t_0,t_1)$. Here, we used \eqref{eqn-v2} in the first inequality of \eqref{eqn-k=1} and the fact that ${e^{\mu\delta}}/{\mu} \leq c$ in the second inequality. Now suppose \eqref{eqn-u2} is true for $t\in [t_0,t_m)$ with $m\geq 1$. We shall prove that \eqref{eqn-u2} holds on $[t_m,t_{m+1})$. To do this, we start with proving that \eqref{eqn-u2} is true for $t=t_m$:
\begin{align}\label{eqn-tm-2}
& v(t_m)e^{\lambda(t_m-t_0)} \cr
\leq & \big\{\rho_1 v(t^-_m) + [\rho_2+(1-\rho_1)\kappa]\sup_{s\in[-r,0]}\{v_1(t^-_m+s)\} \cr
     & + \hat{\chi}(t^-_m) \big\}e^{\lambda(t_m-t_0)} \cr
\leq & \rho_1 u(t^-_m) + e^{\lambda(t_m-t_0)}\hat{\chi}(t^-_m)\cr
     & + [\rho_2+(1-\rho_1)\kappa] e^{\lambda r}\sup_{s\in[-r,0]}\{v(t^-_m+s) e^{\lambda(t_m+s-t_0)}\}\cr
\leq & \rho u(t^-_m) + e^{\lambda(t_m-t_0)}\hat{\chi}(t^-_m)\cr
\leq & u(t_m).
\end{align}
Here, we used conditions (iii) and (iv) in the first inequality of \eqref{eqn-tm-2}, and \eqref{eqn-u2} with the fact that $u(t^-_m+s)\leq u(t^-_m)$ for all $s\in[-r,0]$ in the estimate of the third inequality. For $t\in(t_m,t_{m+1})$, we conclude from \eqref{eqn-v2} and the third inequality of \eqref{eqn-tm-2} that
\begin{align}\label{eqn-vt}
& v(t)e^{\lambda(t-t_0)}\cr
\leq & \rho e^{(\mu+\lambda)(t-t_m)} u(t^-_m) + e^{(\mu+\lambda)(t-t_m)} e^{\lambda(t_m-t_0)} {\chi}(\|w(t^-_m)\|)\cr
     & + \frac{e^{\mu\delta}}{\mu}e^{\lambda(t-t_0)} \hat{\chi}(t^-_m).
\end{align}
In view of \eqref{eqn-u2}, we then get from \eqref{eqn-vt} that
\begin{align}\label{eqn-vt2}
& v(t)e^{\lambda(t-t_0)}\cr
\leq & \rho e^{(\mu+\lambda)\delta} M \alpha(\|\varphi\|) + \big( c \rho e^{\mu\delta} +\frac{e^{\mu\delta}}{\mu} \big) e^{\lambda(t-t_0)} \bar{\chi}(t) \cr
     & + \rho e^{(\mu+\lambda)\delta} M \sum_{t_o<t_k<t_m} e^{\lambda(t_k-t_0)} \hat{\chi}(\|w(t^-_k)\|)\cr
     & + e^{(\mu+\lambda)\delta} M e^{\lambda(t_m-t_0)} \hat{\chi}(\|w(t^-_m)\|)\cr
\leq & u(t), 
\end{align}
i.e., \eqref{eqn-u2} hold on $(t_m,t_{m+1})$. By the method of induction, we conclude that \eqref{eqn-u2} is true for all $t\geq t_0$. The ISS estimate from \eqref{eqn-u2} is similar to that in the proof of Theorem \ref{th1}, and global existence of the solution to \eqref{sys} follows from this estimate.
\end{proof}

\begin{remark}\label{remark-th3}
Theorem \ref{th3} provides sufficient conditions for ISS of system \eqref{sys} with delay-dependent impulses. However, ISS analysis of this type of systems cannot be conducted based on the existing results in \cite{WHC-WXZ:2009,JL-XL-WCX:2011,XMS-WW:2012} due to the existence of time-delay in the impulses. Furthermore, for system \eqref{sys} with delay-free impulses, our result with $\rho_2=0$ is less conservative in the sense that the upper bound of $\delta$ is required to be smaller than $\frac{1}{\mu} \ln(\frac{1}{\rho_1+(1-\rho_1)\kappa})$, while Theorem 1 in \cite{XMS-WW:2012} requires $\delta<\frac{1}{\mu} \ln(\frac{1}{\rho_1+\kappa})$.
\end{remark}

\begin{remark}\label{remark-th3-1}
Compared with Condition (iii) of Theorem \ref{th1}, Condition (iii) of Theorem \ref{th3} is a weaker requirement on the Lyapunov function at each impulse moment. Therefore, the conclusion of Theorem \ref{th3} still holds if we replace this condition with condition (iii) of Theorem \ref{th1}, that is, replace $\chi(\sup_{s\in[-r,0]}\|w(t^-+s)\|)$ with $\chi(\|w(t^-)\|)$ in the condition. However, Theorem \ref{th3} is more applicable to systems in the form of \eqref{sys} with impulses only depending on states at history moments rather than the states at the impulse times. This issue will be demonstrated with Example \ref{example2}. {  Though the discussion in this remark is also applicable to Theorem \ref{th1}, condition (iii) of Theorem \ref{th3} is not necessary for Theorem \ref{th1} (see Example \ref{example1} for illustration).}
\end{remark}

%\begin{corollary}\label{coro}
%Assume that conditions (ii), (iii), and (v) of Theorem \ref{th3} hold and condition (i) is replaced with the following
%\begin{itemize}
%\item[(i)'] $w_1 \|x\|^p \leq V_1(t,x)\leq w_2 \|x\|^p$ and $0\leq V_2(t,\phi)\leq w_3 \|\phi\|^p_{\tau}$,
%\end{itemize}
%then system (\ref{sys}) is uniformly ISS over $\ell_{\textrm{sup}}(\delta)$.
%\end{corollary}

{ 
\begin{remark}\label{remark-coro}
Specify the class $\mathcal{K}_{\infty}$ functions in condition (i) of the above theorem: $\alpha_1(t)=w_1 t^p$, $\alpha_2(t)=w_2 t^p$, and $\alpha_3(t)=w_3 t^p$ with positive constants $w_1$, $w_2$, $w_3$, and $p$, then condition (iv) of Theorem \ref{th3} holds with $\kappa=\frac{w_3}{w_1}$. If we further assume that there is no external input to system (\ref{sys}) (i.e., $w(t)=0$), then Theorem \ref{th3} serves as an global exponential stability result for system (\ref{sys}) with zero input and can be applied to study the delay-dependent impulsive control problems of time-delay systems (see Example \ref{example2} in Section \ref{Sec4} for demonstrations). Our requirement on the upper bound of $\delta$ in condition (v) of Theorem \ref{th3} is less conservative than the one of $\delta<\frac{1}{\mu} \ln(\frac{1}{\rho_1+\rho_2+\kappa})$ in Theorem 3.1 of \cite{XL-KZ-WCX:2016}. 
%Actually, with the given class $\mathcal{K}_{\infty}$ functions in Corollary \ref{coro}, all the ISS results obtained in this paper can be applied to analyze the exponential stability of system \eqref{sys} with zero input $w(t)=0$.
\end{remark}
}

The last ISS result is for system \eqref{sys} with both stable continuous and discrete dynamics, which shows that system \eqref{sys} is ISS for arbitrary impulse sequences.
\begin{theorem}
\label{th4}
Assume that there exist $V_1\in \nu_0$, $V_2\in \nu^*_0$, functions $\alpha_1,\alpha_2,\alpha_3,\chi\in\mathcal{K}_{\infty}$ and constants $\mu\geq 0$, { $\kappa>0$}, $1>\rho_1\geq 0$ and $\rho_2\geq 0$, such that, for all $t\in \mathbb{R}^+$, $x\in\mathbb{R}^n$, $y\in\mathbb{R}^m$ and $\phi\in \mathcal{PC}$, conditions (i), (ii) and (iii) of Theorem \ref{th1} { and condition (iv) of Theorem \ref{th3}} are satisfied and $\rho_1+\rho_2+(1-\rho_1)\kappa< 1$. Then system (\ref{sys}) is uniformly ISS over $\ell_{\textrm{all}}$.
\end{theorem}

\begin{proof}
First, we consider the scenario of $\mu>0$. Since $\rho_1+\rho_2<1$, there exists a small enough constant $\lambda>0$ such that $\rho_1+[\rho_2 +(1-\rho_1)\kappa ]e^{\lambda r}\leq 1$ and $\lambda<\mu$. Then we can pick a large enough constant $c>0$ such that $-\mu+1/c+\lambda\leq 0$. By mathematical induction, we will prove \eqref{eqn-u} holds for all $t\geq t_0$. Similar to the proof of Theorem \ref{th1}, we can show that \eqref{eqn-u} is true on $[t_0,t_1)$. Then, suppose \eqref{eqn-u} holds on $[t_0,t_m)$ with $m\geq 1$. We will show \eqref{eqn-u} is true for $t\in [t_m,t_{m+1})$. When $t=t_m$, we can revise the estimate in \eqref{eqn-tm} as what we did in \eqref{eqn-tm-2} with $\hat{\chi}(t^-_m)$ replaced with $\chi(\|w(t^-_m)\|)$.
%\begin{align}\label{eqn-tm3}
%& v(t_m) e^{\lambda(t_m-t_0)}\cr
%\leq & \rho_1 v(t^-_m) e^{\lambda(t_m-t_0)} + \rho_2 e^{\lambda r} \sup_{-r\leq s\leq 0}\{v(t^-_m+s) e^{\lambda(t_m+s-t_0)}\} \cr
%     & + \chi(\|w(t^-_m)\|) e^{\lambda(t_m-t_0)}\cr
%\leq & (\rho_1+\rho_2 e^{\lambda r}) u(t^-_m) + \chi(\|w(t^-_m)\|) e^{\lambda(t_m-t_0)}\cr
%  =  & u(t_m),
%\end{align}
%which means \eqref{eqn-tm} is true for $t=t_m$. For the second inequality of \eqref{eqn-tm3}, we used the fact that $v(t^-_m+s) e^{\lambda(t_m+s-t_0)}\leq u(t^-_m+s) \leq u(t^-_m)$ for all $s\in [-r,0]$. 
Identical to the discussion in the proof of Theorem \ref{th1}, we can show \eqref{eqn-tm} is true for $t\in(t_m,t_{ m+1})$ by contradiction. The rest of the proof are essentially the same as that of Theorem \ref{th1} and thus omitted.

{ If $\mu=0$, then condition (ii) of Theorem \ref{th3} is true, i.e., $\mathrm{D}^+ V(t,\phi)\leq \varepsilon V(t,\phi) + \chi(\|w(t)\|)$ for any $\varepsilon>0$. We can conclude from Theorem \ref{th3} that system (\ref{sys}) is uniformly ISS over $\ell_{\textrm{sup}}(\delta)$ with 
$$\delta<\frac{\ln[\rho_1+\rho_2+(1-\rho_1)\kappa]}{-\varepsilon}.$$
Let $\varepsilon\rightarrow 0$ in the above inequality and we can see that system (\ref{sys}) is uniformly ISS over $\ell_{\textrm{all}}$.}
\end{proof}

{ 
\begin{remark}
In the above theorems, Lyapunov candidate $V$ comprises a function $V_1$ and a functional $V_2$. Since $V_2$ is indifferent to the impulses, it is necessary to incorporate the function $V_1$ into the Lyapunov candidate to capture the impulse effects on it. Condition (iii) of Theorem \ref{th1} (or Theorem \ref{th3}) characterizes the effects of the state jumps on $V_1$. If $\rho_1<1$, then condition (iv) in Theorem \ref{th3} is necessary to quantify the growth or decay of the Lyapunov candidate $V$ at each impulse moment. Condition (ii) of Theorem \ref{th1} (or Theorem \ref{th3}) describes the unstable (or stable) continuous-time dynamics. Condition (iv) of Theorem \ref{th1} (or condition (v) of Theorem \ref{th3}) balances the continuous-time dynamics and the impulse effects to guarantee the convergence of $V$ in the ISS sense, and then the convergence of the state $x$ can be derived by using the positive definite and decrescent properties of $V_1$ and $V_2$ in condition (i). It is worth mentioning that system \eqref{sys} is in the general form of nonlinear impulsive time-delay systems and it has been studied in \cite{WHC-WXZ:2009,JL-XL-WCX:2011,XMS-WW:2012,SD-MK-AM-LN:2012} for the ISS properties. However, the results in \cite{WHC-WXZ:2009,JL-XL-WCX:2011,XMS-WW:2012} are generally applicable to systems with delay-independent impulses (e.g., $\Delta x(t_k)=I_k(t_k,x(t_k^-),w(t_k^-))$ for $k\in \mathbb{N}$). The reason is that the condition $V_1(t,\phi(0)+I_k(t,\phi,w(t^-)))\leq \rho_1 V_1(t^-,\phi(0))$ in \cite{WHC-WXZ:2009,JL-XL-WCX:2011,XMS-WW:2012} actually is only verifiable for non-delayed impulses. The Lyapunov-Krasovskii-type result obtained in \cite{SD-MK-AM-LN:2012} requires an explicit relationship between $V(t_k)$ and $V(t^-_k)$ which is difficult to quantify. Because the impulses can only affect the function part $V_1$ not the functional part $V_2$. We overcome this difficulty by analyzing the impulse effects on the function part and requiring certain comparison between $V_1$ and $V_2$ when it is necessary, so that the impulse effects on the whole Lyapunov candidate can be obtained.
\end{remark}
}

\section{Examples}\label{Sec4}

{ In this section, we present two examples with numerical simulations. To demonstrate the advantage of Theorem \ref{th1} over the existing results for time-delay systems with delay-independent impulse effects, the first example is slighted adopted from \cite{WHC-WXZ:2009,JL-XL-WCX:2011} with distributed time-delay considered in the impulses.}

\begin{example}\label{example1}
Consider the following time-delay system with distributed-delay dependent impulses:
\begin{subequations}\label{e1sys}
\begin{align}
\label{e1sys-a}\dot{x}(t)&=-\mathrm{sat}(x(t))+a \mathrm{sat}(x(t-\tau))+b \mathrm{sat}(w(t)), ~t\not=t_k,\\
\label{e1sys-b}\Delta x(t)&=\frac{1}{4}\mathrm{sat}(\int^t_{t-\tau}x(s)\mathrm{d}s)+\frac{1}{4}\mathrm{sat}(w(t^-)),~t=t_k,
\end{align}
\end{subequations}
where $a=0.2$, $b=0.1$, $\tau=1$, and $\mathrm{sat}(x)$ is the saturation function defined as $\mathrm{sat}(x)=\frac{1}{2}(|x+1|-|x-1|)$.
\end{example}

To study the ISS property of system \eqref{e1sys}, we choose Lyapunov functional $V(t,\phi)=V_1(t,\phi(0))+V_2(t,\phi)$ with

\begin{align*}
V_1(t,x) &=\left\{\begin{array}{ll}
x^2, & |x|\leq 1, \cr
e^{2(|x|-1)}, & |x|>1,
\end{array}\right.\cr
V_2(t,\phi)&=|a|\int^0_{\tau}\mathrm{sat}^2(\phi(s))\Big(\epsilon+1+\frac{\epsilon s}{\tau}\Big)\mathrm{d}s,
\end{align*}
where $\epsilon>0$. Clearly, condition (i) of Theorem \ref{th1} holds. Similar to the discussion in Example 4.1 of \cite{JL-XL-WCX:2011}, we have $\mathrm{D}^+V(t,\phi)\leq -\mu V(t,\phi)+ \frac{|b|^2}{2}w^2(t)$ with $\mu=\min\{2-(\epsilon+2)|a|-2|b|,\frac{\epsilon}{(\epsilon+1)\tau}\}$. Therefore, condition (ii) of Theorem \ref{th1} is satisfied. Next, we verify condition (iii) of Theorem \ref{th1} at each impulse time. To to this, we consider the following two scenarios:

If $|x(t_k)|\leq 1$, then $V_1(t,x(t_k)) = x^2(t_k) \leq 3x^2(t^-_k)+\frac{3}{16}\mathrm{sat}^2(\int^{t_k}_{t_k-\tau}x(s)\mathrm{d}s)+\frac{3}{16}\mathrm{sat}^2(w(t^-_k))$. Furthermore, if $|x(t^-_k)|\leq 1$, then $x^2(t^-_k)=V_1(t^-_k,x(t^-_k))$; if $|x(t^-_k)|> 1$, then $x^2(t^-_k)\leq e^{2(|x(t^-_k)|-1)}=V_1(t^-_k,x(t^-_k))$. Thus, $x^2(t^-_k)\leq V_1(t^-_k,x(t^-_k))$, and similarly, we can obtain $x^2(t^-_k+s)\leq V_1(t^-_k+s,x(t^-_k+s))$ for any $s\in[-\tau,0]$, which then implies $\mathrm{sat}(\int^{t_k}_{t_k-\tau}x(s)\mathrm{d}s)\leq (\int^{t_k}_{t_k-\tau}x(s)\mathrm{d}s)^2\leq \tau\int^{t_k}_{t_k-\tau}x^2(s)\mathrm{d}s\leq \tau\sup_{s\in[-\tau,0]}\{V_1(t^-_k+s,x(t^-_k+s))\}$. Based on the above discussion, we can have $V_1(t_k,x(t_k))\leq 3V_1(t^-_k,x(t^-_k))+\frac{3\tau}{16}\sup_{s\in[-\tau,0]}\{V_1(t^-_k+s,x(t^-_k+s))\} +\frac{3}{16}\mathrm{sat}^2(w(t^-_k))$.

If $|x(t_k)|> 1$, then
\begin{align*}
& V_1(t_k,x(t_k)) = e^{2(|x(t_k)|-1)}\cr
& \leq e^{2(|x(t^-_k)|+\frac{1}{4}|\mathrm{sat}(\int^{t_k}_{t_k-\tau}x(s)\mathrm{d}s)|+\frac{1}{4}|\mathrm{sat}(w(t^-_k))|-1)}\cr
& \leq e^{2|x(t^-_k)|-1}\cr
& \leq \left\{\begin{array}{ll}
e e^{2|x(t^-_k)|-2}=eV_1(t^-_k,x(t^-_k)), & |x(t^-_k)| > 1, \cr
2e |x(t^-_k)|^2=2eV_1(t^-_k,x(t^-_k)), & \frac{1}{2}< |x(t^-_k)|\leq 1,
\end{array}\right.
\end{align*}
where we used the fact that $e^{2x-2}\leq 2x^2$ for $x\in(\frac{1}{2},1]$.

In either of the scenario, we have shown that condition (iii) holds with $\rho_1=2e$ and $\rho_2=\frac{3\tau}{16}$. According to Theorem \ref{th1}, if $\delta>\ln(\rho_1+\rho_2 e^{\mu\tau})/\mu$, then system \eqref{e1sys} is uniformly ISS over $\ell_{\textrm{inf}}(\delta)$. For illustration, choosing $\epsilon=5$, we can compute $\mu=0.4$ and then $\ln(\rho_1+\rho_2 e^{\mu\tau})/\mu=2.06$, which gives the condition: $\delta>2.06$. Simulation results for system \eqref{e1sys} with the given parameters and $\delta=2.1$ are shown in Fig. \ref{fig1}.

\begin{figure}[!t]
\centering
\includegraphics[width=3.4in]{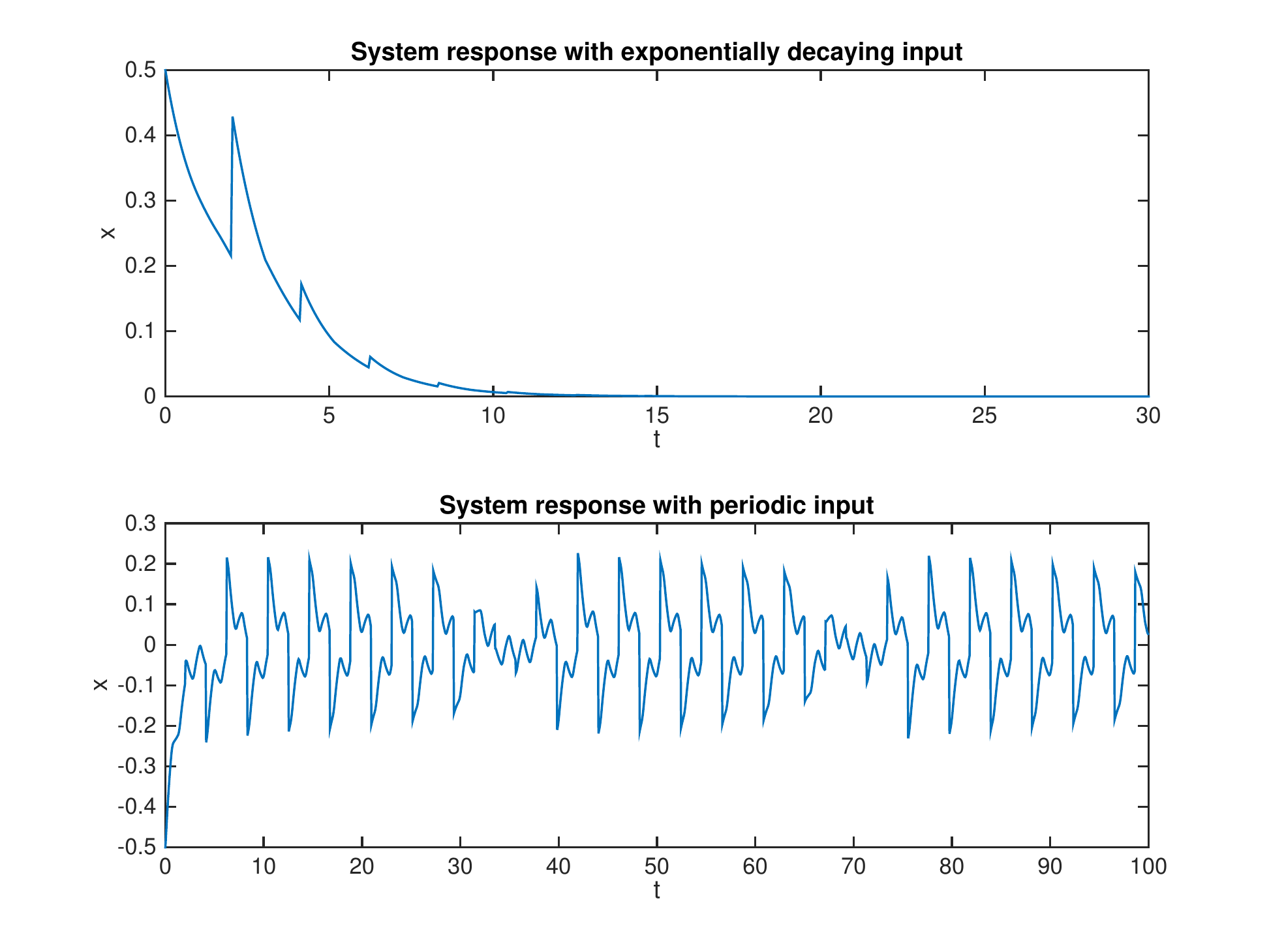}
\caption{Simulation results for system \eqref{e1sys} with exponentially decaying input $w(t)=5e^{-t}$ and periodic input $w(t)=2\sin(14\pi t)$.}
\label{fig1}
\end{figure}

\vskip3mm

{ The second example studies the effects of disturbance inputs on impulsive synchronization of time-delay systems. Applications of the analysis conducted in this example can be extended to impulsive synchronization and stabilization of complex dynamical networks with time-delay (see, e.g., \cite{XL-KZ-WCX:2016,JS-QLH-XJ:2008}).}

\begin{example}\label{example2}
Consider the following two time-delay systems:
\begin{eqnarray}\label{e2sys1}
\dot{x}(t)=Ax(t)+g(x(t-r)),
\end{eqnarray}
and
\begin{eqnarray}\label{e2sys2}
\left\{\begin{array}{ll}
\dot{y}(t)=Ay(t)+g(y(t-r))+Bw(t), & t\not=t_k,\cr
\Delta y(t)=C[y(t-d)-x(t-d)]+Dw(t^-), & t=t_k,
\end{array}\right.
\end{eqnarray}
where $x,y\in\mathbb{R}^n$, $A,C\in\mathbb{R}^{n\times n}$, and $B,D\in\mathbb{R}^{n\times m}$; $r$ and $d$ represent the delays in the continuous and discrete dynamics, respectively; $w(t)$ is the disturbance input; $f$ satisfies Lipschitz condition, that is, there exists a positive constant $L$ such that $\|g(x)-g(y)\|\leq L\|x-y\|$ for any $x,y\in\mathbb{R}^n$. Without loss of generality, we assume that $t_k-t_{k-1}=\delta$ for all $k\in \mathbb{N}$.
\end{example}

Denote the error state $e(t):=y(t)-x(t)$, the dynamics of which can be described by the following impulsive time-delay system:
\begin{eqnarray}\label{e2sys3}
\left\{\begin{array}{ll}
\dot{e}(t)=Ay(t)+\hat{g}(e(t-r))+Bw(t), & t\not=t_k,\cr
\Delta e(t)=Ce(t-d)+Dw(t^-), & t=t_k,
\end{array}\right.
\end{eqnarray}
where $\hat{g}(e(t-r))=g(y(t-r))-g(x(t-r))$. The objective of this example is to study the synchronization between systems \eqref{e2sys1} and \eqref{e2sys2} in the sense that $\lim_{t\rightarrow\infty} e(t)=0$, and the impact of the disturbance input $w(t)$ on this system synchronization, i.e., ISS properties of error system \eqref{e2sys3}. To do so, consider a Lyapunov functional candidate $V(t,e_t)=V_1(t,e(t))+V_2(t,e_t)$ with $V_1(t,e(t))=e^T(t)e(t)$ and $V_2(t,e_t)=\varepsilon L\int^t_{t-r}e^T(s)e(s)\mathrm{d}s$, { then both conditions (i) and (iv) of Theorem \ref{th3} are satisfied with the functions $\alpha_1,\alpha_2,\alpha_3$ specified in Remark \ref{remark-coro} and $p=2$, $w_1=w_2=1$, $w_3=\kappa=\varepsilon r L$}. For simplicity, we denote $v(t)=V(t,e_t)$, $v_1(t)=V_1(t,e(t))$, and $v_2(t)=V_2(t,e_t)$. Similar to the estimation in the proof of Theorem 4.1 from \cite{XL-KZ-WCX:2016}, we have
\begin{align*}
\dot{v}(t) \leq& \mu v(t) +\chi_{\epsilon_1}(\|w(t)\|), ~ t\not=t_k,\cr
v_1(t_k) \leq& \rho_1 v_1(t^-_k)+ \rho_2 \sup_{s\in[-\tau,0]}\{v_1(t^-_k+s)\} \cr
             & +\chi_{\epsilon_2}(\sup_{s\in[-\tau,0]}\|w(t^-_k+s)\|), ~ k\in \mathbb{N},\cr
\end{align*}
where $\mu=\lambda_{max}(A+A^T)+L(\varepsilon+\varepsilon^{-1})+\epsilon_1$, $\rho_1=(1+\xi)\|I+C\|^2$, and $\rho_2=(1+\xi^{-1})[\epsilon_2 d\|C\|(\|A\|+L) +\zeta \|C\|^2]^2$ with $\epsilon_1>0$, $\epsilon_2>1$, and $\xi>0$; $\zeta$ denotes the number of impulses on the interval $(t_k-d,t_k)$; $\chi_{\epsilon_1}$ and $\chi_{\epsilon_2}$ are $\mathcal{K}_{\infty}$ class functions which depend on $\epsilon_2$ and $\epsilon_2$, respectively. It can be seen that conditions (ii) and (iii) of Theorem \ref{th3} are satisfied. Minimizing $\rho_1+\rho_2+(1-\rho_1)\kappa$ with $\xi>0$ , we have $\min_{\xi>0}\{\rho_1+\rho_2+(1-\rho_1)\kappa\}=[\sqrt{1-\kappa}\|I+C\|+\epsilon_2 d \|C\|(\|A\|+L)+\zeta \|C\|^2]^2+\kappa$. If
\begin{eqnarray}\label{e2-1}
&\ln\{[\sqrt{1-\kappa}\|I+C\|+ d \|C\|(\|A\|+L)+\zeta \|C\|^2]^2+\kappa\}\cr
&< -[\lambda_{max}(A+A^T)+L(\varepsilon+\varepsilon^{-1})]\delta
\end{eqnarray}
holds, then there exists positive constants $\epsilon_1$ close to $0$ and $\epsilon_2>1$ close to $1$ such that condition (v) of Theorem \ref{th3} is satisfied. Therefore, the conclusion of Theorem \ref{th3} implies that error system \eqref{e2sys3} is uniformly ISS, provided \eqref{e2-1} holds. As a numerical example, we take $A=\begin{bmatrix}
    -18/7 & 9 & 0 \\
    1 & -1 & 1 \\
    0 & -100/7 & 0
\end{bmatrix}$, $g(x)=\mathrm{sat}(x_1)\begin{bmatrix}
    27/7  \\
    0 \\
    0
\end{bmatrix}$, $B=\begin{bmatrix}
    0  \\
    1/7 \\
    1/7
\end{bmatrix}$, $C=-0.2I$, $D=\begin{bmatrix}
    2/7  \\
    0 \\
    0
\end{bmatrix}$, and $r=2d=0.02$. It can be verified that \eqref{e2-1} is satisfied with $\varepsilon=1$, $L=27/7$, $\delta=0.01$, and $\zeta=0$. According to { Theorem \ref{th3}}, system \eqref{e2sys3} is uniformly ISS. With the given parameters, system \eqref{e2sys1} is a delayed Chua's circuit which exhibits chaotic behaviors as shown in \cite{JL-XL-WCX:2011}. Simulation results for error system \eqref{e2sys3} are shown in Fig. \ref{fig2}. Synchronization between systems \eqref{e2sys1} and \eqref{e2sys2} is demonstrated in Fig. \ref{fig2a} with $w(t)=0$, and the ISS properties of error system \eqref{e2sys3} are illustrated in Fig. \ref{fig2b} and \ref{fig2c} with an exponentially decaying input and a periodic input, respectively.

\begin{figure}[!t]
\centering
\subfigure[]{\label{fig2a}\includegraphics[width=3.4in]{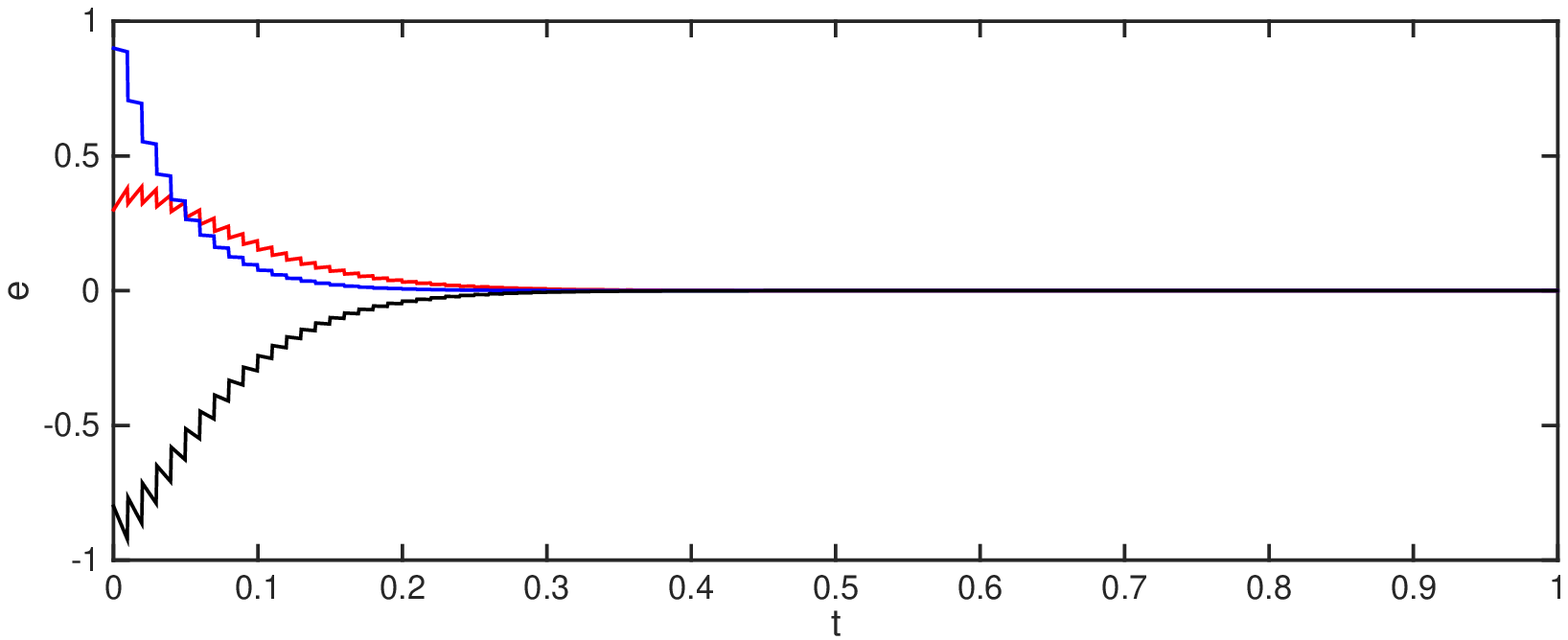}}
\subfigure[]{\label{fig2b}\includegraphics[width=3.4in]{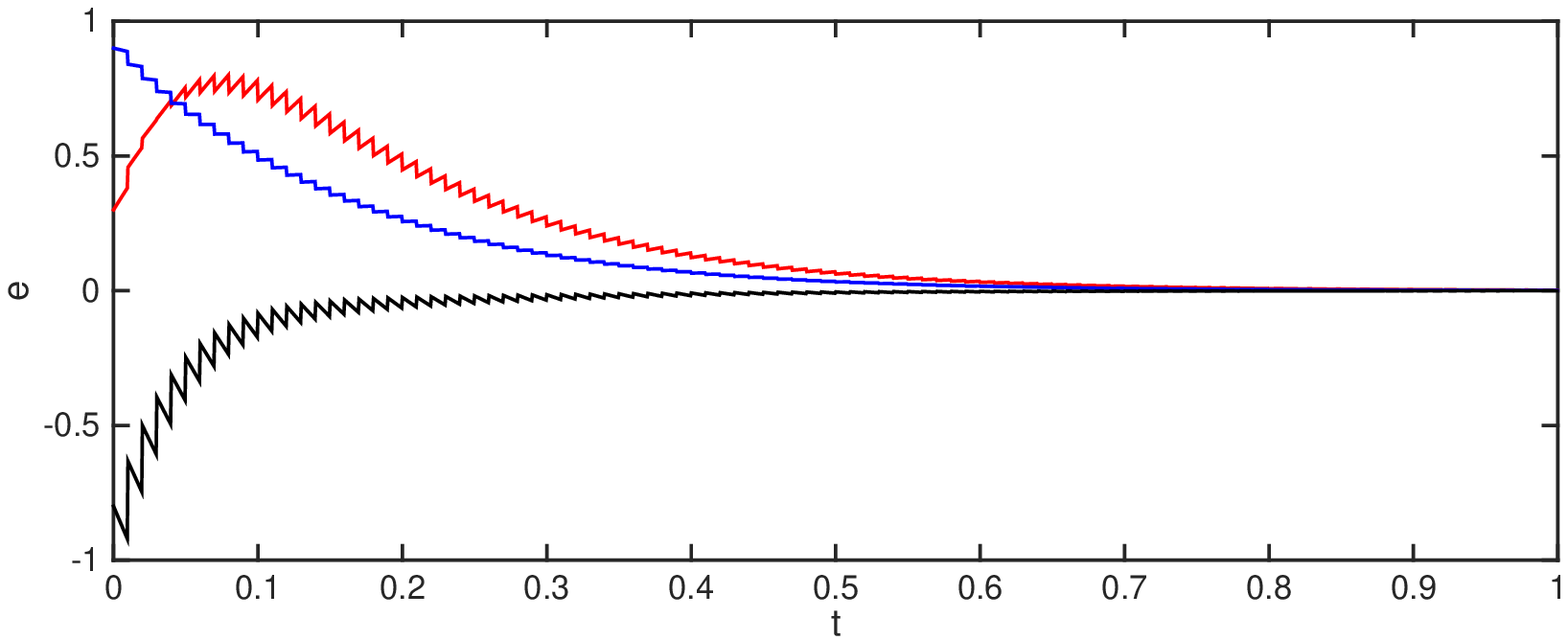}}
\subfigure[]{\label{fig2c}\includegraphics[width=3.4in]{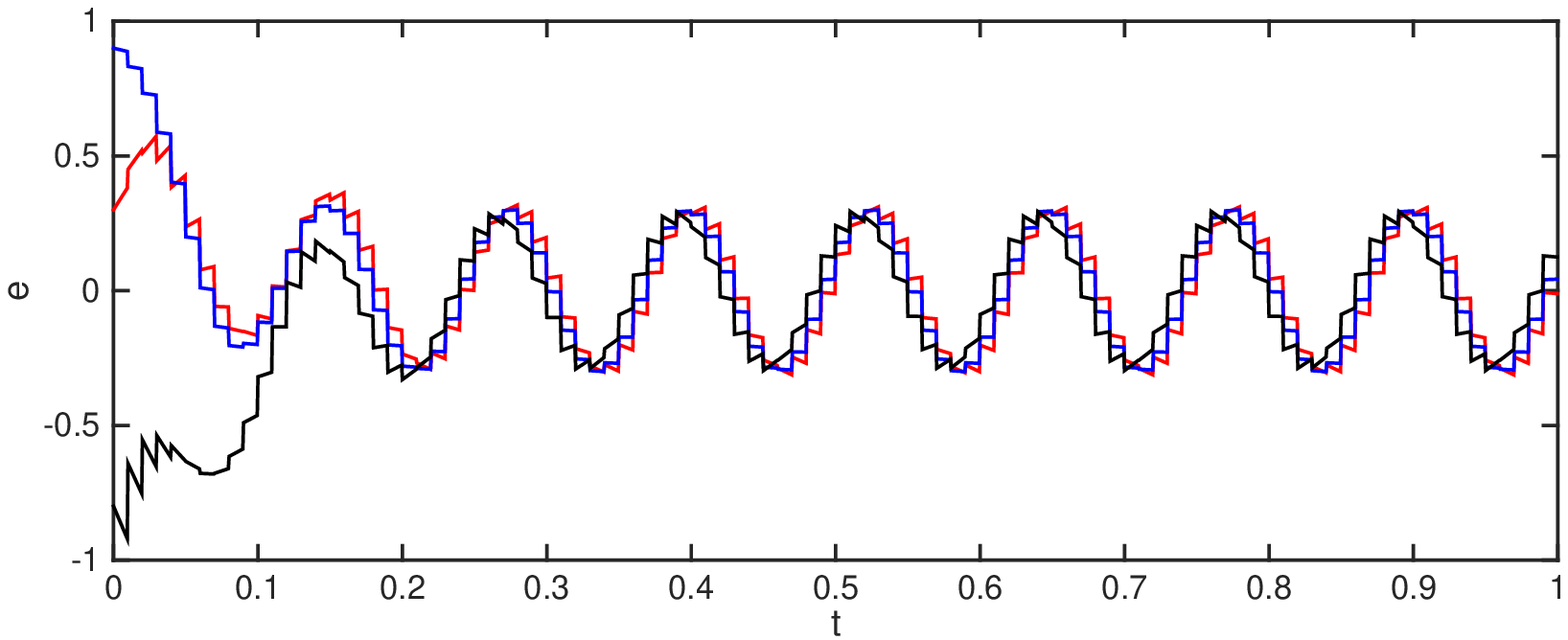}}
\caption{Simulation results for Example \ref{example2} with (a) zero input $w(t)=0$; (b) exponentially decaying input $w(t)=e^{-7t}$; (c) periodic input $w(t)=\cos(-16\pi t)$.}
\label{fig2}
\end{figure}

{ 
\begin{remark}
The derivation of $\rho_1$ and $\rho_2$ is based on the relationship between states $e(t^-_k)$ and $e(t_k-d)$ which is obtained by integrating both sides of \eqref{e2sys3} from $t_k-d$ to $t^-_k$ (see \cite{XL-KZ-WCX:2016} for more details). Therefore, the external input over the time interval $[t_k-d,t_k)$ is brought into the estimation of $V_1$ at the impulse time $t_k$ as discussed in Remark \ref{remark-th3-1}. On the one hand, the delay-dependent impulses synchronize the systems in the sense of ISS. This positive effects of the delayed state $e(t_k-d)$ in the impulses is captured through parameter $\rho_1$. On the other hand, the existence of time-delay in the impulses leads to smaller upper bound for $\delta$ (i.e., $\delta<\frac{-\ln[\rho_1+\rho_2+(1-\rho_1)\kappa]}{\mu}$ with positive $\rho_2$), which is a drawback of the delayed states quantified by parameter $\rho_2$. Actually, the positive effect of the delay-dependent impulses is balanced with this disadvantage in condition (v) of Theorem \ref{th3} so that the synchronization can be achieved.
\end{remark}
}

\section{Conclusions}\label{Sec5}
The method of Lyapunov functionals has been used to investigate ISS property of time-delay systems with delay-dependent impulses. Sufficient conditions have been constructed for systems with ISS continuous dynamics and destabilizing discrete dynamics, unstable continuous dynamics and ISS discrete dynamics, and stable continuous dynamics with zero control input and ISS discrete dynamics, respectively.

Extension of our results to hybrid systems with both switching and impulse effects can be conducted by using the method of multiple Lyapunov-Krasovskii functionals along the line of \cite{JL-XL-WCX:2011,XMS-WW:2012}, which is a topic for future research. Another topic is to investigate integral-input-to-state stability (iISS) of nonlinear systems with delay-dependent impulses. Weaker conditions on the Lyapunov functional candidates are expected since iISS is a relaxed form of ISS as discussed in \cite{WHC-WXZ:2009,JL-XL-WCX:2011,XMS-WW:2012}.
{ Improvements of the current results is also a direction of the future work. For example, less conservative results on dwell time $\delta$ might be available by following the line of \cite{SD-AM:2013}, and the results could be applicable to systems with a larger class of impulse sequences by using the average impulsive interval approach as discussed in \cite{SD-MK-AM-LN:2012,XW-YT-WZ:2016}}.

% if have a single appendix:
%\appendix[Proof of the Zonklar Equations]
% or
%\appendix  % for no appendix heading
% do not use \section anymore after \appendix, only \section*
% is possibly needed

% use appendices with more than one appendix
% then use \section to start each appendix
% you must declare a \section before using any
% \subsection or using \label (\appendices by itself
% starts a section numbered zero.)
%

%\appendices
%\section{Proof of the First Zonklar Equation}
%Appendix one text goes here.
%
%% you can choose not to have a title for an appendix
%% if you want by leaving the argument blank
%\section{}
%Appendix two text goes here.
%
%
%% use section* for acknowledgment
%\section*{Acknowledgment}
%
%
%The authors would like to thank...

% Can use something like this to put references on a page
% by themselves when using endfloat and the captionsoff option.
\ifCLASSOPTIONcaptionsoff
  \newpage
\fi

\end{document}